\documentclass{birkjour}
\usepackage{graphics}
\usepackage{epsfig}
\usepackage{amsfonts}
\usepackage{amssymb}
\usepackage{amsthm}
\usepackage{newlfont}
\usepackage{slashbox}
 \newtheorem{thm}{Theorem}[section]

 \theoremstyle{definition}
 \newtheorem{defn}[thm]{Definition}
 \theoremstyle{remark}
 
 \newtheorem*{ex}{Example}
 \numberwithin{equation}{section}

\begin{document}

%
%
%
%
%
%
%
%
%

\title[Base Polynomials for Schultz Invariants of Linear Phenylenes]
 {\begin{center} Base Polynomials for Schultz Invariants of Linear Phenylenes \end{center}}

\author{Abdul Rauf Nizami}
\address{Faculty of Information Technology, University of Central Punjab, Lahore-Pakistan}
\email{arnizami@ucp.edu.pk}
\author{Muhammad Aslam Malik}
\address{Department of Mathematics, University of the Punjab, Lahore-Pakistan}
\email{aslam.math@pu.edu.pk}
\author{Zahid Mahmood}
\address{Department of Mathematics, University of the Punjab, Lahore-Pakistan}
\email{zahidmahmoodm38@gmail.com}


\maketitle

\begin{abstract}
Let $L_{n}$ be the molecular graph of linear $[n]$phenylene, and  $L'_{n}$ the graph obtained by attaching 4-membered cycles to terminal hexagons of $L_{n-1}$. Thus, $L'_{n}$ is the molecular graph of the $\alpha,\omega$ - dicyclobutadieno derivative of $[n-1]$phenylene, containing $n-1$ hexagons and $n$ squares.
In this paper we give  polynomials which serve as bases for Schultz invariants. Actually, we represent lengths of paths among vertices of degrees 2-2, 2-3, and 3-3 of $L_{n}$ and $L'_{n}$ in terms of polynomials, which are used to find Schultz polynomial, modified Schultz polynomial, Schultz index, and modified Schultz index.
\end{abstract}


\pagestyle{myheadings}
\markboth{\centerline {\scriptsize
 Nizami, Aslam, and Zahid}} {\centerline {\scriptsize
 Base Polynomials for Schultz Invariants of Linear Phenylenes}}
\section{Introduction}

 \noindent A \emph{graph} $G$ is a pair $(V,E)$, where $V$ is the set of vertices and $E$ the set of edges. The edge $e$ between two vertices $u$ and $v$ is denoted by $(u,v)$. The \emph{degree} of a vertex $u$, denoted by $d_{u}$ is the number of edges incident to it. A \emph{path} from a vertex $v$ to a vertex $w$ is a sequence of vertices and edges that starts from $v$ and stops at $w$. The number of edges in a path is the \emph{length} of that path. A graph is said to be \emph{connected} if there is a path between any two of its vertices, as you can see in Figure~\ref{fig1}.\\
\begin{figure}[h]
  \centering
  \includegraphics[width=5cm]{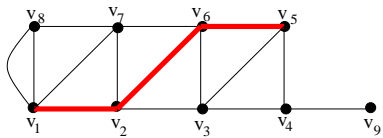}\\
  \caption{A connected graph with a highlighted shortest path from $v_{1}$ to $v_{5}$, and with $d(v_{1})=4$ and $d(v_{5})=3$}\label{fig1}
\end{figure}

\noindent A \emph{molecular graph} is a representation of a chemical compound in terms of graph theory. Specifically, molecular graph is a graph whose vertices correspond to (carbon) atoms of the compound and whose edges correspond to chemical bonds. For instance, the molecular graph of 1-pentene $C_{5}H_{10}$ is given in Figure~\ref{fig2}.

\begin{figure}[h]
\begin{minipage}[h]{12cm}
 \centering
  \includegraphics[width=3.5cm]{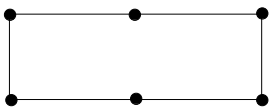}\\
  \caption{Benzene molecule}\label{fig2}
\end{minipage}

\end{figure}

\begin{defn}
\noindent A function $I$ which assigns to every connected graph $G$ a unique number $I(G)$ is called a \emph{graph invariant}. Instead of the function $I$ it is custom to say the number $I(G)$ as the invariant. An invariant of a molecular graph which can be used to determine structure-property or structure-activity correlation is called the \emph{topological index}. A topological index is said to be \emph{degree (distance)} based if it depends on degrees (distance) of the vertices of the graph.
\end{defn}

\noindent In 1989 Harry Schultz introduced the Schultz index in \cite{Schultz:89} and was further studied in \cite{Schultz:00,Gutman:94,Dobrynin:99}.
\begin{defn}\cite{Schultz:89}
The Schultz index of $G$ is defined as
$$S(G) = \frac{1}{2}\sum_{\{v,u\}\in V} (\delta_{u}+\delta_{v}){d(u,v)}$$
Here $d_{u}$ is the degree of the vertex $u$ and $u\neq v$.
\end{defn}
\noindent In 1997 Klavzar and Gutman introduced the modified Schultz index.
\begin{defn} \cite{Klavzar-Gutman:97}
The modified Schultz index of $G$ is defined as
$$MS(G) = \frac{1}{2}\sum_{\{v,u\}\in V} (\delta_{u}\delta_{v}){d(u,v)}$$
Here $d_{u}$ is the degree of the vertex $u$ and $u\neq v$.
\end{defn}

\noindent The Schultz and modified Schultz polynomials were introduced by Gutman in 2005 and found some relations of these polynomials with Wiener polynomial of trees \cite{Gutman:05}.

\begin{defn}\cite{Gutman:05}
The Schultz and modified Schultz polynomials of $G$ are defined respectively as:
\begin{eqnarray*}
  H_{1}(G,x) &=& \frac{1}{2}\sum_{\{v,u\}\in V} (\delta_{u}+\delta_{v})x^{d(u,v)} \\
  H_{2}(G,x) &=& \frac{1}{2}\sum_{\{v,u\}\in V} (\delta_{u}\delta_{v})x^{d(u,v)}
\end{eqnarray*}
Here $d_{u}$ is the degree of the vertex $u$ and $u\neq v$.
\end{defn}
\noindent In 2006 Sen-Peng et. al. did a similar work for hexagonal chains \cite{sen-peng:06}. In 2008 Eliasi and Taeri gave Schultz polynomials of some composite graphs \cite{eliasi-taeri:08}.\\

\begin{defn}
The \emph{linear $[n]$-phenylene} $L_{n}$ and $\alpha,\omega-$dicyclobutadieno derivative of $[n-1]$-phenylene $L'_{n}$ are defined respectively in Figure~\ref{fig3} and Figure~\ref{fig4}:
\begin{figure}[h]
  \centering
  \includegraphics[width=8cm]{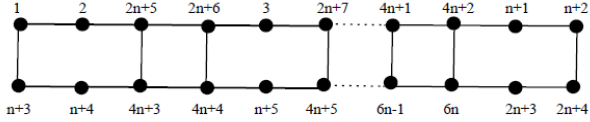}\\
  \caption{linear $[n]$-phenylene, $L_{n}$}\label{fig3}
\end{figure}
\begin{figure}[h]
  \centering
  \includegraphics[width=8cm]{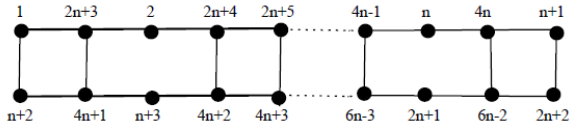}\\
  \caption{$\alpha,\omega-$dicyclobutadieno derivative of $[n-1]$phenylene, $L'_{n}$}\label{fig4}
\end{figure}
\end{defn}

\noindent Ali and Ahmed computed Hosoya polynomials of the pentachains \cite{Ali:11}, and Gutman et al gave generalized Wiener indices of zigzagging pentachains  \cite{Gutman:07}. In this article we give base polynomials for Schultz invariants of linear $[n]$-phenylene and its derivative.

\section{The Results}
Throughout this article by $D^{p,q}$ we shall mean the matrix that represents the distances, and by $H_{b}^{p,q}$ we shall mean the base polynomial that counts the paths of different lengths among vertices of degrees $p$ and $q$ of $L_{n}$ and $L'_{n}$.

\noindent \textbf{A Notational Digression}: By $\#(A_{i})$ we shall mean the number of times $A_{i}$ appears in $D$, and by $\#(k)\in A_{i}$ we shall mean the number of times $k$ appears in $A_{i}$.\\

\subsection{Linear $[n]$-Phenylene}

\begin{thm}\label{thm2.1} Let $n\geq 2$. Then
\begin{eqnarray*}
  H^{2,2}_{b}(L_{n})) &=& 6x^{1}+2\sum_{k=2}^{n}(n-k+3)x^{3k-2}+4\sum_{k=1}^{n-1}x^{3k-1}\\
  &&+6x^{3n-1}+2\sum_{k=2}^{n}(n-k)x^{3k}+2x^{3n}.
   \end{eqnarray*}
\end{thm}
\begin{proof}
The matrix $D_{L_{n}}^{2,2}$ is symmetric with order $(2n+4) \times (2n+4)$. Each row of $D_{L_{n}}^{2,2}$ represents the distances from a vertex $v_{i}$ to the vertices $v_{1},v_{2},\ldots, v_{2n+4}$, respectively. The general form of $D_{L_{n}}^{2,2}$ is
 $$D_{L_{n}}^{2,2}=\left( \begin{array}{cc}
  A_{0}       & A_{1}     \\
  A_{1}   & A_{0}     \\
   \end{array}
\right),$$

\noindent where $A_{0}$ and $A_{1}$ are submatrices of order $n+2\times n+2$ and are:

$$A_{0}=\left(\scriptsize \begin{array}{cccccccccc}
  0 &  1   & 4  & 7 & 10  & 13  & \cdots & 3n-5 & 3n-2 & 3n-1\\
    &  0   & 3  & 6 & 9 & 12  &  \cdots & 3n-6 & 3n-3 & 3n-1\\
    &      & 0  & 3  & 6 & 9 &   \cdots & 3n-9 & 3n-6 & 3n-5\\
    &      &    & \ddots   &   &  &   \vdots &  & & \vdots\\
    &      &    &    &   &  &  0  & 3 &6 & 7\\
     &      &    &    &   &  &    & 0 &3 & 4\\
       &      &    &    &   &  &    &  &0 & 1\\
         &      &    &    &   &  &    &  & & 0\\
   \end{array}
\right)$$
$$A_{1}=\left(\scriptsize \begin{array}{cccccccc}
  1   & 2  & 5 & 8  & \cdots & 3n-4 & 3n-1 & 3n\\
  2   & 3  & 4 & 7  & \cdots & 3n-5 & 3n-2 & 3n-1\\
  5   & 4  & 3 & 4  & \cdots & 3n-8 & 3n-5 & 3n-4\\
  \vdots & &   &    & \vdots &     &     &\vdots\\
  3n-4   & 3n-5  & 3n-8 & 3n-11  & \cdots & 3 &4 & 5\\
  3n-1   & 3n-2  & 3n-5 & 3n-8  & \cdots & 4 &3 & 2\\
  3n   & 3n-1  & 3n-4 & 3n-7  & \cdots & 5 &2 & 1\\
  \end{array}
\right)$$

\noindent Now we give the counts of distinct paths among vertices of degree 2 in $L{n}$: \\

\noindent $c_{1} = \#(1)\in A_{0}\times \#(A_{0})+\#(1)\in A_{1}\times \#(A_{1})=(2)(2)+(2)(1)=6$\\
\noindent $c_{2} = \#(2)\in A_{0}\times \#(A_{0})+\#(2)\in A_{1}\times \#(A_{1})=(0)(2)+(4)(1)=4$\\
\noindent $c_{3} = \#(3)\in A_{0}\times \#(A_{0})+\#(3)\in A_{1}\times \#(A_{1})=(n-1)(2)+(n)(1)=3n-2$\\
\noindent $c_{4} = \#(4)\in A_{0}\times \#(A_{0})+\#(4)\in A_{1}\times \#(A_{1})=(2)(2)+(2n-2)(1)=2n+2$\\
\noindent $c_{5} = \#(5)\in A_{0}\times \#(A_{0})+\#(5)\in A_{1}\times \#(A_{1})=(0)(2)+(4)(1)=4$\\
\noindent $c_{6} = \#(6)\in A_{0}\times \#(A_{0})+\#(6)\in A_{1}\times \#(A_{1})=(n-2)(2)+(0)(1)=2n-2$\\
\noindent $c_{3n-3} = \#(3n-3)\in A_{0}\times \#(A_{0})+\#(3n-3)\in A_{1}\times \#(A_{1})=(1)(2)+(0)(1)=2$\\
\noindent $c_{3n-2} = \#(3n-2)\in A_{0}\times \#(A_{0})+\#(3n-2)\in A_{1}\times \#(A_{1})=(2)(2)+(2)(1)=6$\\
\noindent $c_{3n-1} = \#(3n-1)\in A_{0}\times \#(A_{0})+\#(3n-1)\in A_{1}\times \#(A_{1})=(1)(2)+(4)(1)=6$\\
\noindent $c_{3n} = \#(3n)\in A_{0}\times \#(A_{0})+\#(3n)\in A_{1}\times \#(A_{1})=(0)(2)+(2)(1)=2$
\end{proof}

\begin{thm}\label{thm2.2} For $n\geq 2$, we have
\begin{eqnarray*}
 H^{2,3}_{b}(L_{n}) &=& 2(2n-2)x^{1}+4\sum_{k=2}^{n-1}(n-k+1)x^{3k-2}+4\sum_{k=1}^{n-1}(2n-2k+1)x^{3k-1}\\
 &&+4\sum_{k=1}^{n-1}(n-k+2)x^{3k}+4x^{3n-2}
\end{eqnarray*}
\end{thm}
\begin{proof}
In this case we have
 $D_{L_{n}}^{2,3}=\left( \begin{array}{cc}
  A_{0}      & A_{1} \\
  A_{1}      & A_{0} \\

   \end{array}
\right),$
where $ A_{0}$ and $A_{1}$ are submatrices of order $n+1\times 2n-2$  and are:

$$A_{0}=\left(\scriptsize \begin{array}{cccccccc}
  2   & 3  & 5 & 6  & \cdots & 3n-6 & 3n-4 & 3n-3\\
  1   & 2  & 4 & 5  & \cdots & 3n-7 & 3n-5 & 3n-4\\
  2   & 1  & 1 & 2  & \cdots & 3n-10 & 3n-8 & 3n-7\\
  \vdots & &   &    & \vdots &     &     &\vdots\\
  3n-7   & 3n-8  & 3n-10 & 3n-11  & \cdots & 1 &1 & 2\\
  3n-4   & 3n-5  & 3n-7 & 3n-8  & \cdots & 4 &2 & 1\\
  3n-3   & 3n-4  & 3n-6 & 3n-7  & \cdots & 5 &3 & 2\\
   \end{array}
\right)$$

$$A_{1}=\left(\scriptsize \begin{array}{cccccccc}
  2   & 3  & 5 & 6  & \cdots & 3n-6 & 3n-4 & 3n-3\\
  1   & 2  & 4 & 5  & \cdots & 3n-7 & 3n-5 & 3n-4\\
  2   & 1  & 1 & 2  & \cdots & 3n-10 & 3n-8 & 3n-7\\
  \vdots & &   &    & \vdots &     &     &\vdots\\
  3n-7   & 3n-8  & 3n-10 & 3n-11  & \cdots & 1 &1 & 2\\
  3n-4   & 3n-5  & 3n-7 & 3n-8  & \cdots & 4 &2 & 1\\
  3n-3   & 3n-4  & 3n-6 & 3n-7  & \cdots & 5 &3 & 2\\
  \end{array}
\right)$$

\noindent Now we go for paths: \\
\noindent $c_{1} = \#(1)\in A_{0}\times \#(A_{0})+\#(1)\in A_{1}\times \#(A_{1})=(2n-2)(2)+(2)(0)=2n-2$\\
\noindent $c_{2} = \#(2)\in A_{0}\times \#(A_{0})+\#(2)\in A_{1}\times \#(A_{1})=(2n)(2)+(2n-2)(2)=8n-4$\\
\noindent $c_{3} = \#(3)\in A_{0}\times \#(A_{0})+\#(3)\in A_{1}\times \#(A_{1})=(2)(2)+(2n)(2)=4n+4$\\
\noindent $c_{4} = \#(4)\in A_{0}\times \#(A_{0})+\#(4)\in A_{1}\times \#(A_{1})=(2n-4)(2)+(2)(2)=4n-4$\\
\noindent $c_{5} = \#(5)\in A_{0}\times \#(A_{0})+\#(5)\in A_{1}\times \#(A_{1})=(2n-2)(2)+(2n-4)(2)=8n-12$\\
\noindent $c_{6} = \#(6)\in A_{0}\times \#(A_{0})+\#(6)\in A_{1}\times \#(A_{1})=(2)(2)+(2n-2)(2)=2n-2$\\
\noindent $c_{3n-5} = \#(3n-5)\in A_{0}\times \#(A_{0})+\#(3n-5)\in A_{1}\times \#(A_{1})=(2)(2)+(2)(2)=8$\\
\noindent $c_{3n-4} = \#(3n-4)\in A_{0}\times \#(A_{0})+\#(3n-4)\in A_{1}\times \#(A_{1})=(4)(2)+(2)(2)=12$\\
\noindent $c_{3n-3} = \#(3n-3)\in A_{0}\times \#(A_{0})+\#(3n-3)\in A_{1}\times \#(A_{1})=(2)(2)+(4)(2)=12$\\
\noindent $c_{3n-2} = \#(3n-2)\in A_{0}\times \#(A_{0})+\#(3n-2)\in A_{1}\times \#(A_{1})=(0)(2)+(2)(2)=4$
\end{proof}

\begin{thm}\label{thm2.3} For $n\geq 2$ we have
\begin{eqnarray*}
  H^{3,3}_{b}(L_{n}) &=& 4(n-1)x^{1}+6\sum_{k=2}^{n-1}(n-k)x^{3k-2}+2\sum_{k=1}^{n-1}(2n-2k-1)x^{3k-1}\\&&+6\sum_{k=1}^{n-2}(n-k-1)x^{3k}.
   \end{eqnarray*}
\end{thm}

\begin{proof}
Here the distance matrix is
$D_{L_{n}}^{3,3}=\left( \begin{array}{cc}
  A_{0}     & A_{1} \\
  A_{1}   & A_{0}\\
   \end{array}
\right),$
where $ A_{0}$ and $A_{1}$ are submatrices of order $2n-2\times 2n-2$, and are:

$$A_{0}=\left(\scriptsize \begin{array}{ccccccccc}
  0   & 1  & 3 & 4  & 6  & \cdots & 3n-8 & 3n-6 & 3n-5\\
      & 0  & 2 & 3  & 5  & \cdots & 3n-9 & 3n-7 & 3n-6\\
      &    & 0 & 1  & 3  & \cdots & 3n-11 & 3n-9 & 3n-8\\
      &    &   & \ddots & & \vdots &   \vdots &\vdots    &\vdots\\

      &    &   & & & 0  &1 &3 & 4\\
      &    &   &  & &    &0 &2 & 3\\
      &    &   &  &&    &  &0 & 1\\
      &    &   &  &&    &  &  & 0\\
   \end{array}
\right)$$

$$A_{1}=\left(\scriptsize \begin{array}{cccccccc}
  1   & 2  & 4 & 5 & \cdots & 3n-7 &3n-5  &3n-4 \\
      2   & 1  & 3 & 4 & \cdots & 3n-8 &3n-6  &3n-5 \\
   4   & 3  & 1 & 2 & \cdots & 3n-10 &3n-8  &3n-7 \\
 5   & 4  & 2 & 1 & \cdots & 3n-11 &3n-9  &3n-8 \\
  \vdots &   &&       & \vdots  &  & &\vdots\\
  3n-7   & 3n-8  & 3n-10 &3n-11& \cdots & 1 & 3  &4  \\
  3n-5  & 3n-6  & 3n-8&3n-9 & \cdots & 3 &1  &2 \\
  3n-4   & 3n-5  & 3n-7 &3n-9& \cdots & 4 &2  &1 \\
   \end{array}
\right)$$

\noindent The following are counts of distinct paths in $L_{n}$. \\
\noindent $c_{1} = \#(1)\in A_{0}\times \#(A_{0})+\#(1)\in A_{1}\times \#(A_{1})=(n-1)(2)+(2n-2)(1)=4n-4$\\
\noindent $c_{2} = \#(2)\in A_{0}\times \#(A_{0})+\#(2)\in A_{1}\times \#(A_{1})=(n-2)(2)+(2n-2)(1)=4n-6$\\
\noindent $c_{3} = \#(3)\in A_{0}\times \#(A_{0})+\#(3)\in A_{1}\times \#(A_{1})=(2n-4)(2)+(2n-4)(1)=6n-12$\\
\noindent $c_{4} = \#(4)\in A_{0}\times \#(A_{0})+\#(4)\in A_{1}\times \#(A_{1})=(n-2)(2)+(4n-8)(1)=6n-12$\\
\noindent $c_{5} = \#(5)\in A_{0}\times \#(A_{0})+\#(5)\in A_{1}\times \#(A_{1})=(n-3)(2)+(2n-4)(1)=4n-10$\\
\noindent $c_{6} = \#(6)\in A_{0}\times \#(A_{0})+\#(6)\in A_{1}\times \#(A_{1})=(2n-6)(2)+(2n-6)(1)=6n-18$\\
\noindent $c_{3n-7} = \#(3n-7)\in A_{0}\times \#(A_{0})+\#(3n-7)\in A_{1}\times \#(A_{1})=(1)(2)+(4)(1)=6$\\
\noindent $c_{3n-6} = \#(3n-6)\in A_{0}\times \#(A_{0})+\#(3n-6)\in A_{1}\times \#(A_{1})=(2)(2)+(2)(1)=6$\\
\noindent $c_{3n-5} = \#(3n-5)\in A_{0}\times \#(A_{0})+\#(3n-5)\in A_{1}\times \#(A_{1})=(1)(2)+(4)(1)=6$\\
\noindent $c_{3n-4} = \#(3n-4)\in A_{0}\times \#(A_{0})+\#(3n-4)\in A_{1}\times \#(A_{1})=(0)(2)+(2)(1)=2$
\end{proof}
\begin{ex}
For $n=5$ the base polynomials come from Theorems~\ref{thm2.1}, \ref{thm2.2}, and \ref{thm2.3}, and are respectively
$H^{2,2}_{b}(L_{5})=6x+4x^{2}+13x^{3}+12x^{4}+4x^{5}+6x^{6}+10x^{7}+4x^{8}+4x^{9}+8x^{10}+4x^{11}+2x^{12}+6x^{13}+6x^{14}+2x^{15}$, $H^{2,3}_{b}(L_{5})=16x+36x^{2}+24x^{3}+16x^{4}+28x^{5}+20x^{6}+12x^{7}+20x^{8}+16x^{9}+8x^{10}+12x^{11}+12x^{12}+4x^{13}$, and $H^{3,3}_{b}(L_{5})=16x+14x^{2}+18x^{3}+18x^{4}+10x^{5}+12x^{6}+12x^{7}+6x^{8}+6x^{9}+6x^{10}+2x^{11}$. So, the Schultz polynomial, modified Schultz polynomial, Schultz index, and modified Schultz index are:
\begin{enumerate}
  \item $H_{1}(L_{5})=4H^{2,2}_{b}+5H^{2,3}_{b}+6H^{3,3}_{b}
      =4\big[6x+4x^{2}+13x^{3}+12x^{4}+4x^{5}+6x^{6}+10x^{7}+4x^{8}+4x^{9}+8x^{10}+4x^{11}+2x^{12}+6x^{13}+6x^{14}+2x^{15}\big]+5\big[
      16x+36x^{2}+24x^{3}+16x^{4}+28x^{5}+20x^{6}+12x^{7}+20x^{8}+16x^{9}+8x^{10}+12x^{11}+12x^{12}+4x^{13}\big]      +6\big[16x+14x^{2}+18x^{3}+18x^{4}+10x^{5}+12x^{6}+12x^{7}+6x^{8}+6x^{9}+6x^{10}+2x^{11}\big] =200x+280x^2+ 280x^3+236x^4+216x^5+196x^6+172x^7+152x^8+132x^9+108x^{10}+88x^{11}+68x^{12}+44x^{13}+24x^{14}+8x^{15}$\\

  \item $H_{2}(L_{5})=4H^{2,2}_{b}+6H^{2,3}_{b}+9H^{3,3}_{b}=264x+358x^2+358x^3+306x^4+274x^5+252x^6+220x^7+190x^8+166x^9+134x^{10}+106x^{11}+80x^{12}+48x^{13}+24x^{14}+8x^{15}$\\

  \item $S(L_{5})=\frac{d}{dx}H_{1}(L_{5})|_{x=1}= 12300$\\

  \item $MS(L_{n})=\frac{d}{dx}H_{2}(L_{5})|_{x=1}= 15260 $
\end{enumerate}
\end{ex}
\subsection{$\alpha,\omega-$Dicyclobutadieno Derivative of $[n-1]$phenylene}

\begin{thm}\label{thm2.4} Let $n\geq 2$. Then
\begin{eqnarray*}
 H^{2,2}_{b}(L'_{n})) &=& 2x^{1}+(3n-1)x^{3}+2\sum_{k=2}^{n-1}(n-k)x^{3k-2}+4\sum_{k=1}^{n-1}x^{3k-1}\\
  &&+2\sum_{k=2}^{n}(n-k+1)x^{3k}+2x^{3n-1}
   \end{eqnarray*}
\end{thm}
\begin{proof}
The matrix that represents lengths of distinct paths among vertices of degree 2 in $L'_{n}$ is
$$D_{L'_{n}}^{2,2}=\left( \begin{array}{cc}
  A_{0}       & A_{1}      \\
  A_{1}   & A_{0}    \\

   \end{array}
\right),$$

\noindent where $A_{0}$ and $ A_{1}$ are submatrices of orders $n+1\times n+1$ and are:
 $$A_{0}=\left(\scriptsize \begin{array}{ccccccccc}
  0   & 2  & 5 & 8  & 11  & \cdots & 3n-7 & 3n-4 & 3n-2\\
      & 0  & 3 & 6  & 9  & \cdots & 3n-9 & 3n-6 & 3n-4\\
      &    & 0 & 3  & 6  & \cdots & 3n-12 & 3n-9 & 3n-7\\
      &    &   & \ddots & & \vdots &   \vdots &\vdots    &\vdots\\

      &    &   & & & 0  &3 &6 & 8\\
      &    &   &  & &    &0 &3 & 5\\
      &    &   &  &&    &  &0 & 2\\
      &    &   &  &&    &  &  & 0\\
   \end{array}
\right)$$
$$A_{1}=\left(\scriptsize \begin{array}{cccccccc}
  1   & 3  & 6 & 9 & \cdots & 3n-6 &3n-3  &3n-1 \\
      3   & 3  & 4 & 7 & \cdots & 3n-8 &3n-5  &3n-3 \\
   6   & 4  & 3 & 4 & \cdots & 3n-11 &3n-8  &3n-6 \\
 9   & 7  & 4 & 3 & \cdots & 3n-14 &3n-11  &3n-9 \\
  \vdots &   &&       & \vdots  &  & &\vdots\\
  3n-6   & 3n-8  & 3n-11 &3n-14& \cdots & 3 & 4  &6  \\
  3n-3  & 3n-5  & 3n-8&3n-11 & \cdots & 4 &3  &3 \\
  3n-1   & 3n-3  & 3n-6 &3n-9& \cdots & 6 &3  &1 \\
   \end{array}
\right)$$

\noindent Now we go for paths: \\
\noindent $c_{1} = \#(1)\in A_{0}\times \#(A_{0})+\#(1)\in A_{1}\times \#(A_{1})=(0)(2)+(2)(1)=2$\\
\noindent $c_{2} = \#(2)\in A_{0}\times \#(A_{0})+\#(2)\in A_{1}\times \#(A_{1})=(2)(2)+(0)(1)=4$\\
\noindent $c_{3} = (\#(3)\in A_{0}\times \#(A_{0})+\#(3)\in A_{1}\times \#(A_{1})=(n-2)(2)+(n+3)(1)=3n-1$\\
\noindent $c_{4} = \#(4)\in A_{0}\times \#(A_{0})+\#(4)\in A_{1}\times \#(A_{1})=(0)(2)+(2n-4)(1)=2n-4$\\
\noindent $c_{5} = \#(5)\in A_{0}\times \#(A_{0})+\#(5)\in A_{1}\times \#(A_{1})=(2)(2)+(4)(0)=4$\\
\noindent $c_{6} = \#(6)\in A_{0}\times \#(A_{0})+\#(6)\in A_{1}\times \#(A_{1})=(n-3)(2)+(4)(1)=2n-2$\\
\noindent $c_{3n-3} = \#(3n-3)\in A_{0}\times \#(A_{0})+\#(3n-3)\in A_{1}\times \#(A_{1})=(0)(2)+(4)(1)=4$\\
\noindent $c_{3n-2} = \#(3n-2)\in A_{0}\times \#(A_{0})+\#(3n-2)\in A_{1}\times \#(A_{1})=(1)(2)+(0)(1)=2$\\
\noindent $c_{3n-1} = \#(3n-1)\in A_{0}\times \#(A_{0})+\#(3n-1)\in A_{1}\times \#(A_{1})=(0)(2)+(2)(1)=2$
\end{proof}
\begin{thm}\label{thm2.5} Let $n\geq 2$ . Then
\begin{eqnarray*}
 H^{2,3}_{b}(L'_{n}) &=& 4nx^{1}+4\sum_{k=2}^{n-1}(n-k+2)x^{3k-2}+4\sum_{k=1}^{n-1}(2n-2k)x^{3k-1}\\
 &&+4\sum_{k=1}^{n-1}(n-k)x^{3k}+4x^{3n-2}
   \end{eqnarray*}
\end{thm}
\begin{proof}
Here
 $D_{L'_{n}}^{2,3}=\left( \begin{array}{cc}
   A_{0}      & A_{1} \\
  A_{1}      & A_{0} \\

   \end{array}
\right),$ where
$$A_{0}=\left(\scriptsize \begin{array}{cccccccc}
  1   & 3  & 4 & 6  & \cdots & 3n-6 & 3n-5 & 3n-3\\
  1   & 1  & 2 & 4  & \cdots & 3n-8 & 3n-7 & 3n-5\\
  4   & 2  & 1 & 1  & \cdots & 3n-11 & 3n-10 & 3n-8\\
  \vdots & &   &    & \vdots &     &     &\vdots\\
  3n-8   & 3n-10  & 3n-11 & 3n-13  & \cdots & 1 &2 & 4\\
  3n-5   & 3n-7  & 3n-8 & 3n-10  & \cdots & 2 &1 & 1\\
  3n-3   & 3n-5  & 3n-6 & 3n-8  & \cdots & 4 &3 & 1\\
   \end{array}
\right)$$
and
$$A_{1}=\left(\scriptsize \begin{array}{cccccccc}
  2   & 3  & 5 & 6  & \cdots & 3n-6 & 3n-4 & 3n-3\\
  1   & 2  & 4 & 5  & \cdots & 3n-7 & 3n-5 & 3n-4\\
  2   & 1  & 1 & 2  & \cdots & 3n-10 & 3n-8 & 3n-7\\
  \vdots & &   &    & \vdots &     &     &\vdots\\
  3n-7   & 3n-8  & 3n-10 & 3n-11  & \cdots & 1 &1 & 2\\
  3n-4   & 3n-5  & 3n-7 & 3n-8  & \cdots & 4 &2 & 1\\
  3n-3   & 3n-4  & 3n-6 & 3n-7  & \cdots & 5 &3 & 2\\
  \end{array}
\right)$$
with order $(n+1)\times (2n-2)$.

\noindent $c_{1} = \#(1)\in A_{0}\times \#(A_{0})+\#(1)\in A_{1}\times \#(A_{1})=(2n)(2)+(0)(2)=4n$\\
\noindent $c_{2} = \#(2)\in A_{0}\times \#(A_{0})+\#(2)\in A_{1}\times \#(A_{1})=(2n-4)(2)+(2n)(2)=8n-8$\\
\noindent $c_{3} = \#(3)\in A_{0}\times \#(A_{0})+\#(3)\in A_{1}\times \#(A_{1})=(2)(2)+(2n-4)(2)=4n-4$\\
\noindent $c_{4} = \#(4)\in A_{0}\times \#(A_{0})+\#(4)\in A_{1}\times \#(A_{1})=(2n-2)(2)+(2)(2)=4n$\\
\noindent $c_{5} = \#(5)\in A_{0}\times \#(A_{0})+\#(5)\in A_{1}\times \#(A_{1})=(2n-6)(2)+(2n-2)(2)=8n-16$\\
\noindent $c_{6} = \#(6)\in A_{0}\times \#(A_{0})+\#(6)\in A_{1}\times \#(A_{1})=(2)(2)+(2n-6)(2)=4n-8$\\
\noindent $c_{3n-4} = \#(3n-4)\in A_{0}\times \#(A_{0})+\#(3n-4)\in A_{1}\times \#(A_{1})=(0)(2)+(4)(2)=8$\\
\noindent $c_{3n-3} = \#(3n-3)\in A_{0}\times \#(A_{0})+\#(3n-3)\in A_{1}\times \#(A_{1})=(2)(2)+(0)(2)=4$\\
\noindent $c_{3n-2} = \#(3n-2)\in A_{0}\times \#(A_{0})+\#(3n-2)\in A_{1}\times \#(A_{1})=(0)(2)+(2)(2)=4$
\end{proof}

\begin{thm}\label{thm2.6} Let $n\geq 2$. Then
\begin{eqnarray*}
  H^{3,3}_{b}(L'_{n}) &=& 2(2n-3)x^{1}+2\sum_{k=2}^{n-1}(3n-3k-1)x^{3k-2}+2\sum_{k=1}^{n-1}(2n-2k-1)x^{3k-1}\\&&+2\sum_{k=1}^{n-1}(3n-3k-2)x^{3k}.
   \end{eqnarray*}
\end{thm}

\begin{proof}
Here the distance matrix is
$D_{L'_{n}}^{3,3}=\left( \begin{array}{cc}
  A_{0}     & A_{1} \\
  A_{1}   & A_{0}\\
   \end{array}
\right),$
where $ A_{0}$ and $A_{1}$ are submatrices of order $2n-2\times 2n-2$, and are:

$$A_{0}=\left(\scriptsize \begin{array}{ccccccccc}
  0   & 2  & 3 & 5  & 6  & \cdots & 3n-7 & 3n-6 & 3n-4\\
      & 0  & 1 & 3  & 4  & \cdots & 3n-9 & 3n-8 & 3n-6\\
      &    & 0 & 2  & 3  & \cdots & 3n-10 & 3n-9 & 3n-7\\
      &    &   & \ddots & & \vdots &   \vdots &\vdots    &\vdots\\

      &    &   & & & 0  &2 &3 & 5\\
      &    &   &  & &    &0 &1 & 3\\
      &    &   &  &&    &  &0 & 2\\
      &    &   &  &&    &  &  & 0\\
   \end{array}
\right)$$

$$A_{1}=\left(\scriptsize \begin{array}{cccccccc}
  1   & 3  & 4 & 6 & \cdots & 3n-6 &3n-5  &3n-3 \\
      3   & 1  & 2 & 4 & \cdots & 3n-8 &3n-7  &3n-5 \\
   4   & 2  & 1 & 3 & \cdots & 3n-9 &3n-8  &3n-6 \\
 6   & 4  & 3 & 1 & \cdots & 3n-11 &3n-10  &3n-8 \\
  \vdots &   &&       & \vdots  &  & &\vdots\\
  3n-6   & 3n-8  & 3n-9 &3n-11& \cdots & 1 & 2  &4  \\
  3n-5  & 3n-7  & 3n-8&3n-10 & \cdots & 2 &1  &3 \\
  3n-3   & 3n-5  & 3n-6 &3n-8& \cdots & 4 &3  &1 \\
   \end{array}
\right)$$

\noindent The following are counts of distinct paths in $L_{n}$. \\
\noindent $c_{1} = \#(1)\in A_{0}\times \#(A_{0})+\#(1)\in A_{1}\times \#(A_{1})=(n-2)(2)+(2n-2)(1)=4n-6$\\
\noindent $c_{2} = \#(2)\in A_{0}\times \#(A_{0})+\#(2)\in A_{1}\times \#(A_{1})=(n-1)(2)+(2n-4)(1)=4n-6$\\
\noindent $c_{3} = \#(3)\in A_{0}\times \#(A_{0})+\#(3)\in A_{1}\times \#(A_{1})=(2n-4)(2)+(2n-2)(1)=6n-10$\\
\noindent $c_{4} = \#(4)\in A_{0}\times \#(A_{0})+\#(4)\in A_{1}\times \#(A_{1})=(n-3)(2)+(4n-8)(1)=6n-14$\\
\noindent $c_{5} = \#(5)\in A_{0}\times \#(A_{0})+\#(5)\in A_{1}\times \#(A_{1})=(n-2)(2)+(2n-6)(1)=4n-10$\\
\noindent $c_{6} = \#(6)\in A_{0}\times \#(A_{0})+\#(6)\in A_{1}\times \#(A_{1})=(2n-6)(2)+(2n-4)(1)=6n-16$\\
\noindent $c_{3n-5} = \#(3n-5)\in A_{0}\times \#(A_{0})+\#(3n-5)\in A_{1}\times \#(A_{1})=(0)(2)+(4)(1)=4$\\
\noindent $c_{3n-4} = \#(3n-4)\in A_{0}\times \#(A_{0})+\#(3n-4)\in A_{1}\times \#(A_{1})=(1)(2)+(0)(1)=2$\\
\noindent $c_{3n-3} = \#(3n-3)\in A_{0}\times \#(A_{0})+\#(3n-3)\in A_{1}\times \#(A_{1})=(0)(2)+(2)(1)=2$
\end{proof}
\begin{ex}
If we take $n=6$, then the base polynomials come from Theorems~\ref{thm2.4}, \ref{thm2.5}, and \ref{thm2.6}, and are respectively
$H^{2,2}_{b}(L'_{5})=2x+4x^{2}+14x^{3}+6x^{4}+4x^{5}+8x^{6}+4x^{7}+4x^{8}+6x^{9}+2x^{10}+4x^{11}+4x^{12}+2x^{13}+2x^{14}$,
$H^{2,3}_{b}(L'_{5})=20x+32x^{2}+16x^{3}+20x^{4}+24x^{5}+12x^{6}+16x^{7}+16x^{8}+8x^{9}+12x^{10}+8x^{11}+4x^{12}+4x^{13}$,
 $H^{3,3}_{b}(L'_{5})=14x+14x^{2}+20x^{3}+16x^{4}+10x^{5}+14x^{6}+10x^{7}+6x^{8}+8x^{9}+4x^{10}+2x^{11}+2x^{12}$. It now follows that the Schultz polynomial, modified Schultz polynomial, Schultz index, and modified Schultz index take the forms:
\begin{enumerate}
  \item $H_{1}(L'_{5})=4H^{2,2}_{b}+5H^{2,3}_{b}+6H^{3,3}_{b}
       =4\big[2x+4x^{2}+14x^{3}+6x^{4}+4x^{5}+8x^{6}+4x^{7}+4x^{8}+6x^{9}+2x^{10}+4x^{11}+4x^{12}+2x^{13}+2x^{14}\big]+5\big[
      20x+32x^{2}+16x^{3}+20x^{4}+24x^{5}+12x^{6}+16x^{7}+16x^{8}+8x^{9}+12x^{10}+8x^{11}+4x^{12}+4x^{13}\big]      +6\big[14x+14x^{2}+20x^{3}+16x^{4}+10x^{5}+14x^{6}+10x^{7}+6x^{8}+8x^{9}+4x^{10}+2x^{11}+2x^{12}\big] =192x+260x^2+ 256x^3+220x^4+196x^5+176x^6+156x^7+132x^8+112x^9+92x^{10}+68x^{11}+48x^{12}+28x^{13}+8x^{14}+8x^{15}$\\

  \item $H_{2}(L'_{5})=4H^{2,2}_{b}+6H^{2,3}_{b}+9H^{3,3}_{b}=254x+334x^2+332x^3+288x^4+250x^5+230x^6+202x^7+166x^8+144x^9+116x^{10}+82x^{11}+58x^{12}+32x^{13}+8x^{14}+8x^{15}$\\

  \item $S(L'_{5})=\frac{d}{dx}H_{1}(L_{5})|_{x=1}= 10392$\\

  \item $MS(L'_{5})=\frac{d}{dx}H_{2}(L_{5})|_{x=1}= 13144 $
\end{enumerate}
\end{ex}

\section{Conclusions}
In this article we followed the divide and conquer rule to find Schultz invariants of $L_{n}$ and $L'_{n}$. It was observed that obtaining the general closed forms of the Schultz invariants directly by definition is extremely difficult. In order to handle the situation we counted the paths among vertices of degrees 2-2, 2-3, and 3-3 separately and represented them in terms of polynomials. These polynomials served as bases for computing Schultz invariants as one can directly find Schultz polynomial, modified Schultz polynomial, Schultz index, and modified Schultz index using them. Finally, we gave two examples to show how these bases actually work.


\end{document}